\documentclass[11pt]{article}

\usepackage{amssymb} 
\usepackage{amsthm} 
\usepackage{latexsym}
\usepackage{amsmath} 
\usepackage[utf8]{inputenc} 
\usepackage{enumitem}

 \let\oldmarginpar\marginpar
\renewcommand\marginpar[1]{\-\oldmarginpar[\raggedleft\footnotesize #1]%
  {\raggedright\footnotesize #1}}

\setitemize{itemsep=3pt,topsep=3pt,parsep=1pt,partopsep=3pt}
\setenumerate{itemsep=3pt,topsep=3pt,parsep=1pt,partopsep=3pt}
\setdescription{itemsep=3pt,topsep=3pt,parsep=1pt,partopsep=3pt}

\newcommand{\n}{\nabla}
\newcommand{\Proj}{\mathcal P}

\newcommand{\F}{\mathcal F}
\newcommand{\I}{\mathcal I}

\theoremstyle{plain} 
\newtheorem{theorem}{Theorem}[section]
\newtheorem{proposition}[theorem]{Proposition}

\theoremstyle{definition}

\newtheorem{lemma}[theorem]{Lemma}

\theoremstyle{remark}
\newtheorem{example}{Example}

\newcommand{\bp}{\begin{proof}\;}
  \newcommand{\ep}{\end{proof}}

\title{Finsler 2-manifolds whose holonomy group is the diffeomorphism
  group of the circle}

\author{Zolt\'an Muzsnay and P\'eter T. Nagy}

\begin{document}

\maketitle

\begin{abstract}
  In this paper we show that the topological closure of the holonomy
  group of a certain class of projectively flat Finsler $2$-manifolds of
  constant curvature is maximal, that is isomorphic to the connected
  component of the diffeomorphism group of the circle. This class of
  $2$-manifolds contains the standard Funk plane of constant negative
  curvature and the Bryant-Shen-spheres of constant positive curvature.
  The result provides the first examples describing completely infinite
  dimensional Finslerian holonomy structures.  
\end{abstract}

\footnotetext{2000 {\em Mathematics Subject Classification:} 53C29,
  53B40, 58D05, 22E65, 17B66}

\footnotetext{{\em Key words and phrases:} holonomy, Finsler geometry,
  groups of diffeomorphisms, infinite-dimensional Lie groups, Lie
  algebras of vector fields.}

\footnotetext{This research was supported by the Hungarian Scientific
  Research Fund (OTKA) Grant K 67617.}

\section{Introduction}
\label{intro}

The notion of the holonomy group of a Riemannian or Finslerian manifold
can be introduced in a very natural way: it is the group generated by
parallel translations along loops.  In contrast to the Finslerian case,
the Riemannian holonomy groups have been extensively studied.  One of
the earliest fundamental results is the theorem of Borel and
Lichnerowicz \cite{BL} from 1952, claiming that the holonomy group of a
simply connected Riemannian manifold is a closed Lie subgroup of the
orthogonal group $O(n)$. By now, the complete classification of
Riemannian holonomy groups is known.

Holonomy theory of Finsler spaces is, however, essentially different
from Riemannian theory, and it is far from being well understood. In
\cite{Mu_Na} we proved that the holonomy group of a Finsler manifold
of nonzero constant curvature with dimension greater than $2$ is not a
compact Lie group. In \cite{MuNa1} we showed that there exist large
families of projectively flat Finsler manifolds of constant curvature
such that their holonomy groups are not finite dimensional Lie groups.
The proofs in the above mentioned papers give estimates for the dimension 
of tangent Lie algebras of the holonomy group and therefore they do not 
give direct information about the infinite dimensional structure of the 
holonomy group.

Until now, perhaps because of technical difficulties, not a single
infinite dimensional Finsler holonomy group has been described.  In this
paper we provide the first such a description: we show that the
topological closure of the holonomy group of a certain class of simply
connected, projectively flat Finsler $2$-manifolds of constant curvature
is not a finite dimensional Lie group, and we prove that its topological
closure is $\mathsf{Diff}^{\infty}_+(\mathbb S^1)$, the connected
component of the full diffeomorphism group of the circle.  This class of
Finsler $2$-manifolds contains the positively complete standard Funk
plane of constant negative curvature (positively complete standard Funk
plane), and the complete irreversible Bryant-Shen-spheres of constant
positive curvature (\cite{Shen4}, \cite{Br2}).  We remark that for every
simply connected Finsler $2$-manifold the topological closure of the
holonomy group is a subgroup of $\mathsf{Diff}^{\infty}_+(\mathbb
S^1)$. That means that in the examples mentioned above, the closed
holonomy group is maximal.  In the proof we use our constructive method
developed in \cite{MuNa1} for the study of Lie algebras of vector fields
on the indicatrix, which are tangent to the holonomy group.  In the
proof we use the constructive method developed in \cite{MuNa1} to study
the Lie algebras of vector fields on the indicatrix which are tangent to
the holonomy group.

\section{Preliminaries}

Throughout this article, $M$ is a $C^\infty$ smooth manifold,
${\mathfrak X}^{\infty}(M)$ is the vector space of smooth vector fields
on $M$ and ${\mathsf{Diff}}^\infty(M)$ is the group of all
$C^\infty$-diffeomorphism of $M$.  The first and the second tangent
bundles of $M$ are denoted by $(TM,\pi ,M)$ and $(TTM,\tau ,TM)$,
respectively.
\\[1ex]
A \emph{Finsler manifold} is a pair $(M,\mathcal F)$, where the norm
function $\F\colon TM \to \mathbb{R}_+$ is continuous, smooth on $\hat T
M \!:= \!TM\!  \setminus\! \{0\}$, its restriction ${\mathcal
  F}_x={\mathcal F}|_{_{T_xM}}$ is a positively homogeneous function of
degree one and the symmetric bilinear form
\begin{displaymath}
  g_{x,y} \colon (u,v)\ \mapsto \ g_{ij}(x, y)u^iv^j=\frac{1}{2}
  \frac{\partial^2 \mathcal F^2_x(y+su+tv)}{\partial s\,\partial t}\Big|_{t=s=0}
\end{displaymath}
is positive definite at every $y\in \hat T_xM$.  \\[1ex]
\emph{Geodesics} of $(M, \mathcal F)$ are determined by a system of
$2$nd order ordinary differential equation $\ddot{x}^i + 2 G^i(x,\dot
x)=0$, $i = 1,...,n$ in a local coordinate system $(x^i,y^i)$ of $TM$,
where $G^i(x,y)$ are given by
\begin{equation}
  \label{eq:G_i}  G^i(x,y):= \frac{1}{4}g^{il}(x,y)\Big(2\frac{\partial
    g_{jl}}{\partial x^k}(x,y) -\frac{\partial g_{jk}}{\partial
    x^l}(x,y) \Big) y^jy^k.
\end{equation}
A vector field $X(t)=X^i(t)\frac{\partial}{\partial x^i}$ along a curve
$c(t)$ is said to be parallel with respect to the associated
\emph{homogeneous (nonlinear) connection} if it satisfies
\begin{equation}
  \label{eq:D}
  D_{\dot c} X (t):=\Big(\frac{d X^i(t)}{d t}+  G^i_j(c(t),X(t))\dot c^j(t)
  \Big)\frac{\partial}{\partial x^i}
  =0, 
\end{equation}
where $ G^i_j=\frac{\partial G^i}{\partial y^j}$.
\\
The \emph{horizontal Berwald covariant derivative} $\nabla_X\xi$ of
$\xi(x,y) = \xi^i(x,y)\frac {\partial}{\partial y^i}$ by the vector
field $X(x) = X^i(x)\frac {\partial}{\partial x^i}$ is expressed locally
by
\begin{equation}
  \label{covder}
  \nabla_X\xi = \left(\frac {\partial\xi^i(x,y)}{\partial x^j} 
    - G_j^k(x,y)\frac{\partial \xi^i(x,y)}{\partial y^k} + 
    G^i_{j k}(x,y)\xi^k(x,y)\right)X^j\frac {\partial}{\partial y^i}, 
\end{equation}
where we denote $G^i_{j k}(x,y) := \frac{\partial G_j^i(x,y)}{\partial
  y^k}$.
\\[1ex]
The \emph{Riemannian curvature tensor} field
\begin{math}
  R_{}\!= \! R^i_{jk}(x,y) dx^j\otimes dx^k \otimes
  \frac{\partial}{\partial x^i}
\end{math}
has the expression
\begin{displaymath}
  R^i_{jk}(x,y) =  \frac{\partial G^i_j(x,y)}{\partial x^k} 
  - \frac{\partial G^i_k(x,y)}{\partial x^j} + 
  G_j^m(x,y)G^i_{k m}(x,y) - G_k^m(x,y)G^i_{j m}(x,y). 
\end{displaymath} 
The manifold has \emph{constant flag curvature} $\lambda\in{\mathbb R}$, if for
any $x\in M$ the local expression of the Riemannian curvature is
\begin{displaymath}
  R^i_{jk}(x,y) = \lambda\big(\delta_k^ig_{jm}(x,y)y^m -
  \delta_j^ig_{km}(x,y)y^m\big).
\end{displaymath}
Assume that the Finsler manifold $(M,\mathcal F)$ is locally
projectively flat. Then for every point $x\in M$ there exists an
\emph{adapted} local coordinate system, that is a mapping $(x^1,\dots
,x^n)$ on a neighbourhood $U$ of $x$ into the Euclidean space $\mathbb
R^n$\!, such that the straight lines of $\mathbb R^n$ correspond to the
geodesics of $(M, \F)$. Then the \emph{geodesic coefficients} are of the
form
\begin{equation}
  \label{eq:proj_flat_G_i}
  G^i \!=\! \Proj y^i, \quad  
  G^i_k \!=\! \frac{\partial\Proj}{\partial y^k}y^i \!
  + \! \Proj\delta^i_k,\quad G^i_{kl} 
  \!=\! \frac{\partial^2\Proj}{\partial y^k\partial y^l}y^i 
  \!+\! \frac{\partial \Proj}{\partial y^k}\delta^i_l \!+\! 
  \frac{\partial \Proj}{\partial y^l}\delta^i_k
\end{equation}
where $\Proj(x,y)$ is a 1-homogeneous function in $y$, called the
\emph{projective factor} of $(M,\F)$. According to Lemma 8.2.1 in
\cite{ChSh} p.155, if $(M \! \subset\! \mathbb R^n, \F)$ is a
projectively flat manifold, then its projective factor can be computed
using the formula
\begin{equation}
  \label{eq:P}
  \Proj(x,y) = \frac{1}{2\F}\frac{\partial\F}{\partial x^i}y^i.
\end{equation}

\begin{example} 
  \label{funk} 
  (P. Funk, \cite{Fk1}, \cite{Fk2}, \cite{Fk3}) The {\bf standard Funk
    manifold} $(\mathbb D^n, \F)$ defined by the metric function
  \begin{equation}
    \label{projective1} \F(x,y) = \frac{\sqrt{|y|^2 - \left(|x|^2|y|^2 - 
          \langle x,y\rangle^2\right)}}{1 - |x|^2}\pm\frac{\langle x,y\rangle}
    {1 - |x|^2} 
  \end{equation}
  on the unit disk $\mathbb D^n \subset \mathbb R^n$ is projectively
  flat with constant flag curvature $\lambda=-\frac14$. Its projective
  factor can be computed using formula (\ref{eq:P}):
  \begin{equation}
    \label{projective2}
    \Proj(x,y) =  \frac12 \;\frac{\pm\sqrt{|y|^2 - \left(|x|^2|y|^2  
          - \langle x,y\rangle^2\right)} + \langle x,y\rangle}{1 - |x|^2}.
  \end{equation}
  We call the standard Funk $2$-manifold the \emph{standard
    Funk plane}.
\end{example}
\begin{example} 
  \label{BS} 
  The {\bf Bryant-Shen sphere}s $(\mathbb S^n,
  \F_\alpha)_{_{|\alpha|<\frac{\pi}{2}}}$, are the elements of a
  $1$-parameter family of projectively flat complete Finsler manifolds
  with constant flag curvature $\lambda=1$ defined on the $n$-sphere
  $\mathbb S^n$.  The metric function and the projective factor at $0\!\in\!
  \mathbb R^n$ have the form
  \begin{equation}\label{BrSh}
    \F(0,y) = |y| \,\cos \alpha, \quad 
    \Proj(0,y) = |y|\,\sin \alpha,\quad \text{with}\quad |\alpha|<\frac{\pi}{2}
  \end{equation}
  in a local coordinate system corresponding to the Euclidean canonical
  coordinates, centered at $0\in \mathbb R^n$. R.~Bryant in [Br1], [Br2]
  introduced and studied this class of Finsler metrics on $\mathbb S^2$
  where great circles are geodesics. Z. Shen generalized its
  construction to $\mathbb S^n$ and obtained the expression (\ref{BrSh})
  (cf. Example 7.1. in \cite{Shen4} and Example 8.2.9 in \cite{ChSh}).

\end{example}

\section{Holonomy group as subgroup of the diffeomorphism group}

The group $\mathsf{Diff}^{\infty}(K)$ of diffeomorphisms of a compact
manifold $K$ is an infinite dimensional Lie group belonging to the class
of Fr\'echet Lie groups. The Lie algebra of $\mathsf{Diff}^{\infty}(K)$
is the Lie algebra ${\mathfrak X}^{\infty}(K)$ of smooth vector fields
on $K$ endowed with the negative of the usual Lie bracket of vector
fields. $\mathsf{Diff}^{\infty}(K)$ is modeled on the locally convex
topological Fr\'echet vector space ${\mathfrak X}^{\infty}(K)$. A
sequence $\{f_j \}_{j\in\mathbb N}\subset{\mathfrak X}^{\infty}(K)$
converges to $f$ in the topology of ${\mathfrak X}^{\infty}(K)$ if and
only if the functions $f_j$ and all their derivatives converge uniformly
to $f$, respectively to the corresponding derivatives of $f$. We note
that the difficulty of the theory of Fr\'echet manifolds comes from the
fact that the inverse function theorem and the existence theorems for
differential equations, which are well known for Banach manifolds, are
not true in this category. These problems have led to the concept of
regular Fr\'echet Lie groups (cf. H. Omori \cite{Omori} Chapter III,
A. Kriegl -- P. W. Michor \cite{KrMi} Chapter VIII). The distinguishing
properties of regular Fr\'echet Lie groups can be summarized as the
existence of smooth exponential map from the Lie algebra of the
Fr\'echet Lie groups to the group itself, and the existence of product
integrals, which produces the convergence of some approximation methods
for solving differential equations (cf. Section III.5. in \cite{Omori},
pp. 83 --89). J. Teichmann gave a detailed discussion of these
properties in \cite{TE}. In particular $\mathsf{Diff}^{\infty}(K)$ is a
topological group which is an inverse limit of Lie groups modeled on
Banach spaces and hence it is a regular Fr\'echet Lie group (Corollary
5.4 in \cite{Omori}).
\\[1ex]
Let $H$ be a subgroup of the diffeomorphism group
$\mathsf{Diff}^{\infty}(K)$ of a differentiable manifold $K$. A vector
field $X\!\in\!{\mathfrak X}^{\infty}(K)$ is called \emph{tangent to}
$H\subset\mathsf{Diff}^{\infty}(K)$ if there exists a ${\mathcal
  C}^1$-differentiable $1$-parameter family $\{\Phi(t)\in
H\}_{t\in\mathbb R}$ of diffeomorphisms of $K$ such that
$\Phi(0)=\mathsf{Id}$ and \(\frac{d\Phi(t)}{d t}\big|_{t=0}=X.\) A Lie
subalgebra $\mathfrak h$ of ${\mathfrak X}^{\infty}(K)$ is called
\emph{tangent to} $H$, if all elements of $\mathfrak h$ are tangent
vector fields to $H$.\\[1ex]
We denote by $(\I M,\pi,M)$ the \emph {indicatrix bundle} of the Finsler
manifold $(M,\mathcal F)$, the \emph{indicatrix} $\I_xM$ at $x \in M$ is
the compact hypersurface $\I_xM:= \{y \in T_xM ; \ \mathcal F(y) = 1\}$
in $T_xM$ which is diffeomorphic to the sphere $\mathbb S^{n-1}$, if $\dim(M) =
n$. The homogeneous (nonlinear) parallel translation
$\tau_{c}:T_{c(0)}M\to T_{c(1)}M$ along a curve $c:[0,1]\to M$ preserves
the value of the Finsler function, hence it induces a map
$\tau_{c}\colon \I_{c(0)}M\longrightarrow \I_{c(1)}M$ between the
indicatrices.
\\[1ex]
The \emph{holonomy group} $\mathsf{Hol}_x(M)$ of the Finsler manifold
$(M,\F)$ at a point $x\in M$ is the subgroup of the group of
diffeomorphisms ${\mathsf{Diff}^{\infty}}({\I}_xM)$ generated by
(nonlinear) parallel translations of ${\I}_xM$ along piece-wise
differentiable closed curves initiated at the point $x\in M$.  The
\emph{closed holonomy group} is the topological closure
$\overline{\mathsf{Hol}_x(M)}$ of the holonomy group with respect of the
Fr\'echet topology of ${\mathsf{Diff}^{\infty}}({\I}_xM)$.
\\
We remark that the diffeomorphism group
${\mathsf{Diff}^{\infty}}({\I}_xM)$ of the indicatrix ${\I}_xM$ is a
regular infinite dimensional Lie group modeled on the vector space
${\mathfrak X}^{\infty}({\I}_xM)$. Particularly ${\mathsf
  {Diff}}^\infty(I_xM)$ is a strong inverse limit Banach (ILB) Lie
group. In this category of groups the exponential map can be defined,
and the group structure is locally determined by the Lie algebra
${\mathfrak X}^{\infty}({\I}_xM)$ of the Lie group
${\mathsf{Diff}^{\infty}}({\I}_xM)$ (cf.~~\cite{KrMi}, \cite{Omori}).
\\[1ex]
For any vector fields $X, Y\in {\mathfrak X}^{\infty}(M)$ on $M$ the
vector field $\xi = R(X,Y)\in {\mathfrak X}^{\infty}({\I}M)$ is called a
\emph{curvature vector field} of $(M, \F)$ (see \cite{Mu_Na}). The Lie
algebra $\mathfrak{R}(M)$ of vector fields generated by the curvature
vector fields of $(M, \F)$ is called the \emph{curvature algebra} of
$(M, \F)$.  The restriction $\mathfrak{R}_x(M) \! := \!
\big\{\,\xi\big|_{\I_xM}\ ;\ \xi\in\mathfrak{R}(M)\, \big\} \subset
{\mathfrak X}^{\infty}({\I}_xM)$ of the curvature algebra to an
indicatrix $\I_x M$ is called the \emph{curvature algebra at the point
  $x\in M$}.
\\[1ex]
The \emph{infinitesimal holonomy algebra} of $(M,\mathcal F)$ is the
smallest Lie algebra $\mathfrak{hol}^{*}(M)$ of vector fields on the
indicatrix bundle $\I M$ satisfying the following properties
\begin{enumerate}
\item[a)] any curvature vector field $\xi$ belongs to
  $\mathfrak{hol}^{*}(M)$,
\item[b)] if $\xi, \eta\in \mathfrak{hol}^{*}(M)$ then
  $[\xi,\eta]\in\mathfrak{hol}^{*}(M)$,
\item[c)] if $\xi\in\mathfrak{hol}^{*}(M)$ and $X\in {\mathfrak
    X}^{\infty}(M)$ then the horizontal Berwald covariant derivative
  $\nabla_{\!\! X}\xi$ also belongs to $\mathfrak{hol}^{*}(M)$.
\end{enumerate} 
The restriction
\begin{math}
  \mathfrak{hol}^{*}_x(M) \! := \!  \big\{\,\xi\big|_{\I_xM}\ ;\
  \xi\in\mathfrak{hol}^{*}(M)\, \big\} \subset {\mathfrak
    X}^{\infty}({\I}_xM)
\end{math}
of the infinitesimal holonomy algebra to an indicatrix $\I_x M$ is
called the \emph{infinitesimal holonomy algebra at the point $x\in
  M$}. Clearly, $\mathfrak{R}(M)\subset\mathfrak{hol}^{*}(M)$ and
$\mathfrak{R}_x(M)\subset\mathfrak{hol}^{*}_x(M)$ for any $x\in M$ (see
\cite{MuNa}).
\\[1ex]
Roughly speaking, the image of the curvature tensor (the curvature
vector fields) determines the curvature algebra, which generates (with
the bracket operation and the covariant derivation) the infinitesimal
holonomy algebra. Localising these object at a point $x\in M$ we obtain the
curvature algebra and the infinitesimal holonomy algebra at $x$.
\\[1ex]
The following assertion will be an important tool in the next
discussion:
\\[1ex]
\emph{The infinitesimal holonomy algebra $\mathfrak{hol}^{*}_x(M)$ at
  any point $x\in M$ is tangent to the holonomy group
  $\mathsf{Hol}_x(M)$.}  (Theorem 6.3 in \cite{MuNa}).
\\[1ex]
The topological closure of the holonomy group is an interesting
geometrical object which can reflect the geometric properties of the
Finsler manifold. In the characterization of the closed holonomy group
we use the following
\begin{proposition}  
  \label{expo}
  The group $\big\langle\!\exp(\mathfrak{hol}^{*}_x(M))\!\big\rangle$
  generated by the image $\exp(\mathfrak{hol}^{*}_x(M))$ of the
  infinitesimal holonomy algebra $\mathfrak{hol}^{*}_x(M)$ at a point
  $x\in M$ with respect to the exponential map $\exp :{\mathfrak
    X}^{\infty}({\I}_xM)\to{\mathsf{Diff}^{\infty}}({\I}_xM)$ is a
  subgroup of the closed holonomy group $\overline{\mathsf{Hol}_x(M)}$.
\end{proposition}
\begin{proof} 
  For any element $X\in \mathfrak{hol}^{*}_x(M)$ there exists a
  ${\mathcal C}^1$-differentiable $1$-parameter family $\{\Phi(t)\in
  \mathsf{Hol}_x(M) \}_{t\in\mathbb R}$ of diffeomorphisms of the
  indicatrix $\I_xM$ such that $\Phi(0)=\mathsf{Id}$ and $\frac{d\Phi}{d
    t}\big|_{t=0}=X$.  Then, considering $\Phi(t)$ as "hair" and using
  the argument of Corollary 5.4. in \cite{Omori}, p. 85, we get that
  \begin{math}
    {\Phi^{n}\big(\frac{t}{n}\big)}\!\!=\Phi\big(\frac{t}{n}\big) \circ
    \cdots \circ \Phi \big(\frac{t}{n}\big)
  \end{math}
  in $\mathsf{Hol}_x(M)$ as a sequence of
  $\mathsf{Diff}^{\infty}({\I}_xM)$ converges uniformly in all
  derivatives to $\exp(t X)$.  It follows that we have
  \begin{displaymath}
    \{\exp(tX); t\in\mathbb R\}\subset \overline{\mathsf{Hol}_x(M)}
  \end{displaymath}
  for any $X\!\in\!\mathfrak{hol}^{*}_x(M)$ and therefore 
  \begin{math}
    \exp(\mathfrak{hol}^{*}_x(M))\! \subset \overline{\mathsf{Hol}_x
      (M)}.
  \end{math}
  Naturally, if we consider the generated group, then the relation is
  preserved, that is
  \begin{math}
    \big\langle\!\exp(\mathfrak{hol}^{*}_x(M))\!\big\rangle
    \subset  \overline{\mathsf{Hol}_x (M)},
  \end{math}
  which proves the proposition.
\end{proof}

\section{The group $\mathsf{Diff}_+^{\infty}({\mathbb S}^1)$ and the Fourier
  algebra}

Let $(M, \F)$ be a Finsler $2$-manifold. In this case the indicatrix is
diffeomorphic to $\mathbb S^1$ at any point $x\in M$. If there exists a
non-vanishing curvature vector field at $x\in M$ then any other
curvature vector field at $x\in M$ is proportional to it, which means
that the curvature algebra is at most 1-dimensional.  However, the
infinitesimal holonomy algebra can be an infinite dimensional subalgebra
of $\mathfrak X^{\infty}(\mathbb S^1)$, therefore the holonomy group can
be an infinite dimensional subgroup of $\mathsf{Diff}_+^{\infty}(\mathbb
S^1)$, cf. \cite{MuNa1}.
\\[1ex]
Let $\mathbb S^1 = \mathbb R \mod 2\pi$ be the unit circle with the
standard counterclockwise orientation. The group
$\mathsf{Diff}_+^{\infty}(\mathbb S^1)$ of orientation preserving
diffeomorphisms of the $\mathbb S^1$ is the connected component of
$\mathsf{Diff}^{\infty}(\mathbb S^1)$.  The Lie algebra of
$\mathsf{Diff}_+^{\infty}(\mathbb S^1)$ is the Lie algebra $\!{\mathfrak
  X}^{\infty}(\mathbb S^1)$ -- denoted also by $\mathsf{Vect}(\mathbb
S^1)$ in the literature -- can be written in the form
$f(t)\frac{d}{dt}$, where $f$ is a $2\pi$-periodic smooth functions on
the real line $\mathbb R$. A sequence $\{f_j\frac{d}{dt}\}_{j\in\mathbb
  N}\subset\mathsf{Vect}\, (\mathbb S^1)$ converges to $f\frac{d}{dt}$
in the Fr\'echet topology of $\mathsf{Vect}(\mathbb S^1)$ if and only if
the functions $f_j$ and all their derivatives converge uniformly to $f$,
respectively to the corresponding derivatives of $f$. The Lie bracket on
$\mathsf{Vect}(\mathbb S^1)$ is given by
\begin{displaymath}
  \Big[f\frac{d}{dt},g\frac{d}{dt}\Big] 
  = \Big(g\frac{df}{dt} - \frac{dg}{dt}f\Big)\frac{d}{dt}.
\end{displaymath}
The \emph{Fourier algebra} $\mathsf{F}(\mathbb S^1)$ on $\mathbb S^1$ is the Lie
subalgebra of $\mathsf{Vect}(\mathbb S^1)$ consisting of vector fields
$f\frac{d}{dt}$ such that $f(t)$ has finite Fourier series, i.e. $f(t)$
is a Fourier polynomial. The vector fields $\big\{\frac{d}{dt},\, \cos
nt\frac{d}{dt},\, \sin nt\frac{d}{dt}\big\}_{n\in \mathbb N}$ \ provide
a basis for $\mathsf{F}(\mathbb S^1)$. A direct computation shows that the
vector fields
\begin{equation}
  \label{gensys} 
  \frac{d}{dt},\quad \cos t\frac{d}{dt},\quad 
  \sin t\frac{d}{dt}, \quad \cos 2t\frac{d}{dt},
  \quad \sin 2t\frac{d}{dt}
\end{equation} 
generate the Lie algebra $\mathsf{F}(\mathbb S^1)$.  The complexification
$\mathsf{F}(\mathbb S^1)\!\otimes_{\mathbb R} \! \mathbb C$ of
$\mathsf{F}(\mathbb S^1)$ is called the \emph{Witt algebra} $\mathsf{W}(\mathbb S^1)$
on $\mathbb S^1$ having the natural basis
$\big\{ie^{int}\frac{d}{dt}\big\}_{n\in \mathbb Z}$, with the Lie
bracket
\begin{math}
  [ie^{imt}\frac{d}{dt}, ie^{int}\frac{d}{dt}] = ie^{i(n-m)t}
  \frac{d}{dt}.
\end{math}
\begin{lemma} 
  \label{expoFour} 
  The group $\big\langle \, \overline{\exp(\mathsf{F}(\mathbb S^1))}
  \,\big\rangle$ generated by the topological closure of the exponential
  image of the Fourier algebra $\mathsf{F}(\mathbb S^1)$ is the orientation
  preserving diffeomorphism group $\mathsf{Diff}_+^{\infty}(\mathbb S^1)$.
\end{lemma}
\begin{proof} 
  The Fourier algebra $\mathsf{F}(\mathbb S^1)$ is a dense subalgebra of
  $\mathsf{Vect}(\mathbb S^1)$ with respect to the Fr\'echet topology,
  i.e. $\overline{\mathsf{F}(\mathbb S^1)}=\mathsf{Vect}(\mathbb S^1)$.  This
  assertion follows from the fact that the Fourier series of the
  derivatives of smooth functions are the derivatives of their Fourier
  series (c.f. \cite{SzN}, p. 386,) and from Fej\'er's approximation
  theorem, (c.f. \cite{SzN}, 429,) claiming that any continuous function
  can be approximated uniformly by the arithmetical means of the partial
  sums of its Fourier series. The exponential mapping is
  continuous (c.f.~Lemma 4.1 in \cite{Omori}, p.~79), hence we have
  \begin{equation}
    \label{eq:F_in_Dif_1}
    \exp(\mathsf{Vect}(\mathbb S^1))
    =\exp\big(\overline{\mathsf{F}(\mathbb S^1)}\big) \subset
    \overline{\exp(\mathsf{F}(\mathbb S^1))} 
    \subset \mathsf{Diff}_+^{\infty}(\mathbb S^1)
  \end{equation}
  which gives for the generated groups the relations
  \begin{equation}
    \label{eq:F_in_Dif_2}
    \big\langle\!    \exp(\mathsf{Vect}(\mathbb S^1))
    \!\big\rangle \subset
    \big\langle \, \overline{\exp(\mathsf{F}(\mathbb S^1))} \,\big\rangle
    \subset \mathsf{Diff}_+^{\infty}(\mathbb S^1).
  \end{equation}
  Moreover, the conjugation map 
  \begin{math}
    \mathsf{Ad}:\mathsf{Diff}_+^{\infty}(\mathbb
    S^1)\times\mathsf{Vect}(\mathbb S^1)
  \end{math}
  satisfies the relation $h\exp s\xi \,h^{-1} =\exp s\mathsf{Ad}(h)\xi$
  for every $h\in\mathsf{Diff}_+^{\infty}(\mathbb S^1)$ and
  $\xi\in\mathsf{Vect}(\mathbb S^1)$ (cf. Definition 3.1 in
  \cite{Omori}, p. 9).  Clearly, the Lie algebra $\mathsf{Vect}(\mathbb
  S^1)$ is invariant under conjugation and hence the group
  \begin{math}
    \big\langle\!  \exp(\mathsf{Vect}(\mathbb S^1)) \!\big\rangle
  \end{math}
  is also invariant under conjugation.  Therefore
  \begin{math}
    \big\langle\!  \exp(\mathsf{Vect}(\mathbb S^1)) \!\big\rangle
  \end{math}
  is a non-trivial normal subgroup of $\mathsf{Diff}_+^{\infty}(\mathbb
  S^1)$.  On the other hand $\mathsf{Diff}_+^{\infty}(\mathbb S^1)$ is a
  simple group (cf. \cite{Herman}) which means that its only non-trivial
  normal subgroup is itself. Therefore we have
  \\
  \begin{math}
    \big\langle\!  \exp(\mathsf{Vect}(\mathbb S^1)) \!\big\rangle
    =\mathsf{Diff}_+^{\infty}(\mathbb S^1),
  \end{math}
  and using (\ref{eq:F_in_Dif_2}) we get 
  \begin{math}
    \big\langle \, \overline{\exp(\mathsf{F}(\mathbb S^1))} \,\big\rangle =
    \mathsf{Diff}_+^{\infty}(\mathbb S^1).
  \end{math}
\end{proof} 

\section{Holonomy of the standard Funk plane and the Bryant-Shen
  $2$-spheres}

Using the results of the preceding chapter we can prove the following
statement, which provides a useful tool for the investigation of the
closed holonomy group of Finsler $2$-manifolds.
\begin{proposition} 
  \label{ifcontains} 
  If the infinitesimal holonomy algebra $\mathfrak{hol}^{*}_x(M)$ at a
  point $x\in M$ of a simply connected Finsler $2$-manifold $(M, \F)$
  contains the Fourier algebra $\mathsf{F}(\mathbb S^1)$ on the indicatrix at
  $x$, then $\overline{\mathsf{Hol}_x(M)}$ is isomorphic to
  $\mathsf{Diff}_+^{\infty}(\mathbb S^1)$.
\end{proposition}
\begin{proof} 
  Since $M$ is simply connected we have
  \begin{equation}
    \label{eq:Hol_in_Diff}
    \overline{\mathsf{Hol}_x(M)}\subset\mathsf{Diff}_+^{\infty}(\mathbb S^1).
  \end{equation}
  On the other hand, using Proposition \ref{expo}, we get
  \begin{displaymath}
    \exp(\mathsf{F} (\mathbb S^1))\!\subset\! \overline{\mathsf{Hol}_x(M)}
    \ \Rightarrow\  
    \overline{\exp(\mathsf{F} (\mathbb S^1))}\!\subset\! \overline{\mathsf{Hol}_x(M)}
    \ \Rightarrow\
    \big\langle \, \overline{\exp(\mathsf{F}(\mathbb S^1))} \,\big\rangle
    \!\subset\! \overline{\mathsf{Hol}_x(M)},
  \end{displaymath}
  and from the last relation, using Lemma \ref{expoFour}, we can obtain
  that
  \begin{equation}
    \label{eq:Diff_in_Hol}
    \mathsf{Diff}_+^{\infty}(\mathbb S^1) \subset \overline{\mathsf{Hol}_x(M)}.
  \end{equation}
  Comparing (\ref{eq:Hol_in_Diff}) and (\ref{eq:Diff_in_Hol}) we get the
  assertion.
\end{proof}

Using this proposition we can prove our main result:
\begin{theorem}
  \label{prop_Hol_Dif}
  Let $(M, \mathcal F)$ be a simply connected projectively flat Finsler
  manifold of constant curvature $\lambda\neq 0$. Assume that there
  exists a point $x_0\in M$ such that the following conditions hold
  \vspace{-3pt}
  \begin{enumerate}
  \item[\emph{A)}] the induced Minkowski norm $\mathcal F(x_0,y)$ on
    $T_{x_0}M$ is an Euclidean norm $\lVert y\rVert$,
  \item[\emph{B)}] the projective factor $\mathcal P({x_0},y)$ on
    $T_{x_0}M$ satisfies $\mathcal P({x_0},y)=c\!\cdot\!\lVert y\rVert$
    with $0\neq c\in\mathbb R$. 
  \end{enumerate}
  Then the closed holonomy group $\overline{\mathsf{Hol}_{x_0}(M)}$ at
  $x_0$ is isomorphic to $\mathsf{Diff}^{\infty}_+(\mathbb S^1)$.
\end{theorem}
\begin{proof}
  Since $(M, \F)$ is a locally projectively flat Finsler manifold of
  non-zero constant curvature, we can use an $(x^1,x^2)$ local
  coordinate system centered at $x_0\in M$, corresponding to the
  canonical coordinates of the Euclidean space which is projectively
  related to $(M, \F)$. Let $(y^1,y^2)$ be the induced coordinate system
  in the tangent plane $T_xM$. In the sequel we identify the tangent
  plane $T_{x_0}M$ with $\mathbb R^2$ by using the coordinate system
  $(y^1, y^2)$. We will use the Euclidean norm $\lVert(y^1, y^2)\rVert =
  \sqrt{(y^1)^2+(y^2)^2}$ of $\mathbb R^2$ and the corresponding polar
  coordinate system $(e^r,t)$, too.
  \\[1ex]
  Let us consider the curvature vector field $\xi$ at $x_0= 0$ defined
  by
  \begin{displaymath}
    \xi\!=\!R \left( \frac{\partial}{\partial
        x_1}, \frac{\partial}{\partial x_2} \right)\Big|_{x=0} 
    = \lambda\big(\delta_2^ig_{1m}(0,y)y^m -\delta_1^ig_{2m}(0,y)y^m\big) 
    \frac{\partial}{\partial x^i} 
  \end{displaymath}
  Since $(M, \F)$ is of constant flag curvature, the horizontal Berwald
  covariant derivative $\nabla_WR$ of the tensor field $R$ vanishes.
  Therefore the covariant derivative of $\xi$ can be written in the form
  \begin{displaymath}
    \n_W\xi = R\left(\n_k\left(\frac{\partial}{\partial x^1}
        \wedge\frac{\partial}{\partial x^2}\right)\right)W^k.
  \end{displaymath}
  Since
  \begin{displaymath}
    \n_k\left(\frac{\partial}{\partial x^1} \wedge 
      \frac{\partial}{\partial x^2}\right) = 
    \left(G^1_{k1} + G^2_{k2}\right)\frac{\partial}{\partial x^1}
    \wedge\frac{\partial}{\partial x^2}
  \end{displaymath}
  we obtain $\n_W\xi = \left(G^1_{k1} + G^2_{k2}\right)W^k\xi.$ Using
  (\ref{eq:proj_flat_G_i}) we can express $G^m_{km} = 3\frac{\partial
    P}{\partial y^k} = 3c\frac{y^k}{\lVert y\rVert}$ and hence
  \begin{displaymath}
    \nabla_k \xi = 3\frac{\partial P}{\partial y^k} 
    = 3c\frac{y^k}{\lVert y\rVert}\xi, 
  \end{displaymath}
  where we use the notation $\nabla_k = \nabla_\frac{\partial}{\partial
    x^k}$.  Moreover we have
  \begin{displaymath}
    \nabla_j \left(\frac{\partial \Proj}{\partial y^{k}}\right)
    = \frac {\partial^2 \Proj}{\partial x^{j}\partial y^{k}} 
    - G_j^{m}\frac{\partial^2 \Proj}{\partial y^{m}\partial y^{k}}
    = \frac{\partial^2 \Proj}{\partial x^{j}\partial y^{k}} -
    \Proj\frac{\partial^2 \Proj}{\partial y^{k}\partial
      y^{j}},
  \end{displaymath}
  and hence 
  \[\n_j\left(\n_k\xi\right) = 
  3\left\{\frac {\partial^2 \Proj}{\partial x^{j}\partial y^{k}} - 
    \Proj\frac{\partial^2 \Proj}{\partial y^{k}\partial y^{j}} + \frac{\partial \Proj}
    {\partial y^{k}} \frac{\partial \Proj}{\partial y^{j}}\right\} \xi.\] 
  According to Lemma 8.2.1, equation (8.25) in \cite{ChSh}, p. 155, we obtain  
  \begin{displaymath}
    \frac {\partial^2 \Proj}{\partial x^j\partial y^k} 
    = \frac {\partial \Proj}{\partial y^j}\frac {\partial \Proj}{\partial y^k} + 
    \Proj\,\frac {\partial^2 \Proj}{\partial y^j\partial y^k} -
    \lambda\,\left(\frac {\partial \F}{\partial y^j}\frac {\partial
        \F}{\partial y^k} + \F\,\frac {\partial^2 \F}{\partial y^j\partial
        y^k}\right).
  \end{displaymath}
  It follows
  \begin{displaymath}
    \n_j\left(\n_k\xi\right) = \left((4c^2 - \lambda)\frac{\partial
        \F}{\partial y^{j}} \frac{\partial \F}{\partial y^{k}} + (2c^2 -
      \lambda)\F\frac{\partial^2 \F}{\partial y^{j}\partial y^{k}} \right)
    \xi.
  \end{displaymath}
  Using the conditions (A) and (B) we get
  \[\n_j\left(\n_k\xi\right) = \left(2c^2\,\frac{y^jy^k}{\lVert y\rVert^2} + (2c^2 - \lambda)\delta^{jk}\right) \xi,\]
  where $\delta^{jk}\in\{0,1\}$ such that $\delta^{jk} = 1$ if and only if $j = k$.\\
  Let us introduce polar coordinates $y^1 = r\cos t$, $y^2 = r\sin t$ in
  the tangent space $T_{x_0}M$. We can express the curvature vector
  field, its first and second covariant derivatives along the indicatrix
  curve $\{(\cos t,\sin t);\; 0\leq t<2\pi\}$ as follows:
  \begin{displaymath}
    \xi \!=\! \lambda \frac{d}{dt}, \quad \nabla_1 \xi 
    \!=\! 3c \lambda \cos t  \frac{d}{dt},\quad
    \nabla_2  \xi \!= \!-\! 3c \lambda \sin t \frac{d}{dt}, \quad
    \n_1\! \left(\n_2\xi\right) = c^2\lambda \sin 2t \frac{d}{dt},
  \end{displaymath}
  \begin{displaymath}
    \n_1\!\left(\n_1\xi\right) \!=\! \lambda \left(2c^2\cos^2 t 
      \!+\! 2c^2 \!-\! \lambda\right) \frac{d}{dt},\quad   
    \n_2\!\left(\n_2\xi\right) \!=\!   \lambda \left(2c^2\sin^2 t 
      \!+\! 2c^2 \!-\! \lambda\right) \frac{d}{dt}.
  \end{displaymath}
  Since $c\,\lambda \neq 0$, the vector fields
  \begin{displaymath}
    \frac{d}{dt}, \quad \cos t \,\frac{d}{dt}, \quad \sin t \,
    \frac{d}{dt},\quad \cos t\sin t\,\frac{d}{ dt}, 
    \quad \cos^2 t \,\frac{d}{dt},\quad \sin^2 t\,\frac{d}{ dt}
  \end{displaymath}
  are contained in the infinitesimal holonomy algebra
  $\mathfrak{hol}^{*}_{x_0}(M)$. It follows that the generator system
  \begin{displaymath}
    \left\{\frac{d}{dt}, \quad \cos t \,\frac{d}{dt}, \quad 
      \sin t \,\frac{d}{dt},\quad \cos  2t\,\frac{d}{ dt}, 
      \quad \sin2 t\,\frac{d}{ dt}\right\}
  \end{displaymath}
  of the Fourier algebra $\mathsf{F}(\mathbb S^1)$ (c.f. equation
  (\ref{gensys})) is contained in the infinitesimal holonomy algebra
  $\mathfrak{hol}^{*}_{x_0}(M)$.  Hence the assertion follows from
  Proposition \ref{ifcontains}.
\end{proof}

We remark, that the standard Funk plane and the Bryant-Shen $2$-spheres
are connected, projectively flat Finsler manifolds of nonzero constant
curvature.  Moreover, in each of them, there exists a point $x_0\in M$
and an adapted local coordinate system centered at $x_0$ with the
following properties: the Finsler norm $\mathcal F(x_0,y)$ and the
projective factor $\mathcal P({x_0},y)$ at $x_0$ are given by $\mathcal
F({x_0},y) = \lVert y\rVert$ and by $\mathcal
P({x_0},y)=c\!\cdot\!\lVert y\rVert$ with some constant $c\in\mathbb R$,
$c\neq 0$, where $\lVert y\rVert$ is an Euclidean norm in the tangent
space at $x_0$. Hence we obtain
\begin{theorem}
  The closed holonomy groups of the standard Funk plane and of the
  Bryant-Shen $2$-spheres are maximal, that is diffeomorphic to the
  orientation preserving diffeomorphism group of $\mathbb S^1$.
\end{theorem}

\bigskip

\noindent
Zolt\'an Muzsnay
\\
Institute of Mathematics, University of Debrecen,
\\
H-4032 Debrecen, Egyetem t\'er 1, Hungary
\\
{\it E-mail}: {\tt {}muzsnay@science.unideb.hu}\vspace{4mm}
\\
P\'eter T. Nagy
\\
Institute of Applied Mathematics, \'Obuda University
\\
H-1034 Budapest, B\'ecsi \'ut 96/b, Hungary \\ {\it E-mail}: {\tt
  {}nagy.peter@nik.uni-obuda.hu}

\end{document}